\documentclass[francais]{amsart}
\usepackage[T1]{fontenc}
\usepackage[francais]{babel}
\usepackage{amssymb}
\usepackage{amsmath}
\usepackage{amsthm}
\usepackage{amscd}
\usepackage{amsfonts}
\usepackage{stmaryrd}
\usepackage{pb-diagram}
\usepackage{epic,eepic,epsfig}
\usepackage{a4wide}
\usepackage{enumerate}
\usepackage{nextpage}
\usepackage{fancyhdr}
\usepackage{pstricks}
\usepackage{graphicx,booktabs,psfrag,color,pstcol,pst-grad, pst-eucl, pstricks-add, auto-pst-pdf}
\usepackage{epsf}

\newtheorem{prop}{Proposition}[section]

\newtheorem{lem}[prop]{Lemme}
\newtheorem{thm}[prop]{Th\'eor\`eme}

\theoremstyle{definition}

\newtheorem{ques}[prop]{Question}

\newcommand{\R}{\mathbb{R}}
\newcommand{\Q}{\mathbb{Q}}
\newcommand{\Z}{\mathbb{Z}}
\newcommand{\floor}[1]{\left\lfloor #1\right\rfloor}

\begin{document}

\title{Exponentielle tronqu\'ee et autres contes galoisiens}
\author{{P}atrick {R}abarison, {F}abien {P}azuki et {P}ascal {M}olin}
\address{Patrick Rabarison. Universit\'e d'Antananarivo, D\'epartment de Math\'ematiques et d'Informatique, BP 906 - Antananarivo 101 - Madagascar.}
\email{prabarison@gmail.com}
\address{Fabien Pazuki. University of Copenhagen, Institute of Mathematics, Universitetsparken 5, 2100 Copenhagen, Denmark, and Universit\'e de Bordeaux, IMB, 351, cours de la Lib\'eration, 33400 Talence, France.}
\email{fpazuki@math.ku.dk}
\address{Pascal Molin, Institut de Math\'ematiques de Jussieu - Paris rive gauche
UMR7586, 75013 Paris, France.}
\email{molin@math.univ-paris-diderot.fr}

\thanks{Les auteurs remercient l'IRN GANDA (CNRS) pour le soutien, ainsi que l'Universit\'e d'Antananarivo pour l'hospitalit\'e. Ils remercient aussi Jean-Fran\c{c}ois Mestre, Farbod Shokrieh et Laurent Berger pour leurs pr\'ecieuses remarques. Merci \`a l'arbitre pour ses remarques utiles. FP et PM sont soutenus par le projet ANR-17-CE40-0012 Flair. FP est soutenu par le projet ANR-20-CE40-0003 Jinvariant.}
\maketitle

\begin{abstract}
Nous faisons un tour d'horizon de r\'esultats sur les groupes de Galois de troncatures de s\'eries enti\`eres, avec la s\'erie exponentielle comme figure de proue. On propose ensuite de nouvelles explorations avec des calculs explicites de groupes de Galois d'approximants de Pad\'e qui semblent jouir, eux aussi, de propri\'et\'es int\'eressantes.
\end{abstract}

{\flushleft
\textbf{Mots-Clefs:} Groupes de Galois. S\'eries enti\`eres. Approximants de Pad\'e.\\
}

\begin{center}
---------
\end{center}

\begin{center}
\textbf{Truncated exponential and other tales of Galois groups}.
\end{center}

\textsc{Abstract.} We give a survey of results on the Galois group of polynomials obtained by truncation of power series, the main example being the exponential series. We also present some evidence of a new phenomena\!: Galois groups of Pad\'e approximation polynomials seem to have special properties as well.

\begin{center}
---------
\end{center}

{\flushleft
\textbf{Keywords\!\!:} Galois groups. Power series. Pad\'e approximation.\\
\textbf{Mathematics Subject Classification\!\!:} 11R32, 12F12. }

\thispagestyle{empty}

\section{Introduction}

Le calcul explicite de groupes de Galois de polyn\^omes \`a coefficients rationnels est une entreprise passionnante et bien souvent difficile. La r\'esolution du probl\`eme de Galois inverse pour le groupe des permutations $\mathcal{S}_n$ et son groupe altern\'e $\mathcal{A}_n$ est classique et date de Hilbert, on pourra consulter \cite{JLY} pour un expos\'e plus moderne. Une id\'ee diff\'erente circulait d\'ej\`a en 1929 : prendre une s\'erie enti\`ere \`a coefficients rationnels, la tronquer \`a l'ordre $N$ et chercher des informations galoisiennes sur le polyn\^ome obtenu. Est-il irr\'eductible ? Son groupe de Galois est-il particulier ? Peut-on le calculer ?

Schur a ouvert la voie avec les articles \cite{Sch1, Sch2} qui traitent de la s\'erie exponentielle $$\mathrm{e}^x=\sum_{k=0}^{+\infty}\frac{x^k}{k!}.$$ \noindent Pour tout $n\in{\mathbb{N}}$, notons l'exponentielle tronqu\'ee \`a l'ordre $n$ par
\begin{equation}\label{expoN}
T_n(x)=\sum_{k=0}^{n}\frac{x^k}{k!}.
\end{equation}

\noindent C'est un polyn\^ome en $x$ de degr\'e $n$, \`a coefficients rationnels. Est-il irr\'eductible ? Quel est le groupe de Galois de $T_n$ ? Des calculs explicites men\'es en PARI/gp pour des petites valeurs de $n$ font appara\^itre une r\'egularit\'e qui aiguisera la curiosit\'e du lecteur. On notera $\mathcal{S}_n$ le groupe sym\'etrique d'indice $n$ et $\mathcal{A}_n$ son sous-groupe altern\'e form\'e des permutations paires. Le th\'eor\`eme de Schur est le suivant.

\begin{thm}\label{premier}
L'exponentielle tronqu\'ee v\'erifie les propri\'et\'es suivantes.
\begin{enumerate}
\item[(a)] Soit $n$ un nombre entier non divisible par $4$. Le groupe de Galois du polyn\^ome $T_n$ est $\mathcal{S}_n$. 
\item[(b)] Soit $k\geq 1$ un entier naturel. Le groupe de Galois du polyn\^ome $T_{4k}$ est $\mathcal{A}_{4k}$.
\end{enumerate}
\end{thm}

C'est bien entendu la raison principale qui pousse \`a formuler la question classique suivante : 

\begin{ques}\label{ques tronc}
Quelles s\'eries enti\`eres ont des troncatures qui jouissent de propri\'et\'es galoisiennes similaires \`a celles propos\'ees dans le th\'eor\`eme \ref{premier}?
\end{ques}

Un coup d'oeil \`a la section \ref{horizon} permet de comprendre que les familles de polyn\^omes orthogonaux sont jusqu'ici des acteurs importants dans cette pi\`ece, mais on verra aussi dans les sections \ref{Pade} et \ref{sec:real_cyclo} que les polyn\^omes impliqu\'es dans la construction des approximations de Pad\'e, en lieu et place des simples troncatures, pourrait s'av\'erer devenir une source d'exemples d'une nouvelle nature.

Ce premier texte de Schur a \'et\'e suivi rapidement par \cite{Sch3} qui traite de polyn\^omes de Laguerre $L_n$ d\'efinis pour tout entier $n\geq0$ par $$L_n(x)=\sum_{k=0}^{n}\binom{n}{k}\frac{(-x)^k}{k!}$$ (notons que \cite{Sch3} traite aussi de la s\'erie exponentielle), puis par \cite{Sch4} qui porte encore sur les polyn\^omes de Laguerre, et aborde de plus les polyn\^omes de Hermite $H_n$, d\'efinis pour tout entier $n\geq0$ par $$H_n(x)=\sum_{k=0}^{\lfloor n/2 \rfloor}(-1)^k \binom{n}{2k} 1\cdot3\cdot5\cdot\ldots\cdot(2k-1) x^{n-2k}.$$

Notre objectif ici est double. Nous pr\'esentons tout d'abord un tour d'horizon de travaux plus r\'ecents sur le m\^eme th\`eme dans la section \ref{horizon}. Nous pr\'esentons ensuite en section \ref{Pade} une \'etude bas\'ee sur des calculs men\'es en PARI/gp et qui indique que d'autres fonctions naturelles, dont certains approximants de Pad\'e, jouissent de propri\'et\'es galoisiennes int\'eressantes, notamment les s\'eries enti\`eres associ\'ees aux fonctions $x\mapsto\frac12\log\left(\frac{1+x}{1-x}\right)$ et $x\mapsto\sin(x) + \sinh(x)$. Ces r\'esultats num\'eriques poussent les auteurs \`a formuler la question suivante :

\begin{ques}\label{ques pad}
Quelles fonctions ont des approximants de Pad\'e qui jouissent de propri\'et\'es galoisiennes particuli\`eres ?
\end{ques}

Nous concluons cette introduction avec les deux r\'esultats suivants, en r\'eponse partielle \`a la question \ref{ques pad}. Le premier, le th\'eor\`eme \ref{expo pad}, concerne les propri\'et\'es galoisiennes des approximants de Pad\'e de la fonction exponentielle. Le second, le th\'eor\`eme \ref{thm:real_cyclo}, est nouveau et concerne la s\'erie enti\`ere associ\'ee \`a la fonction $x\mapsto(1+4x)^{-1/2}$ au voisinage de $0$.

 Nos premiers r\'esultats num\'eriques, rendus publics dans une premi\`ere version de ce texte en 2020, portaient \`a croire que les groupes de Galois des num\'erateurs et d\'enominateurs des fractions rationnelles des approximants de Pad\'e de la fonction exponentielle avaient pour groupes de Galois $\mathcal{S}_n$ ou $\mathcal{A}_n$. Cette observation a depuis \'et\'e confirm\'ee partiellement par \cite{CuSh}, voici leur th\'eor\`eme (les d\'efinitions n\'ecessaires sont rappel\'ees en section \ref{Pade}).

\begin{thm}\label{expo pad}
Soient $P(m,k,x)$ et $Q(m,k,x)$ les approximants de Pad\'e d'ordre $(m,k)$ de la s\'erie exponentielle.
\begin{enumerate}
\item[(a)] Pour tout $m\geq1$, les polyn\^omes $P(m,m,x)$ et $Q(m,m,x)$ sont irr\'eductibles et ont pour groupe de Galois $\mathcal{S}_m$.
\item[(b)] Supposons que $P(m,m+1,x)$ et $Q(m,m+1,x)$ soient irr\'eductibles. Alors le groupe de Galois de $P(m,m+1,x)$ est $\mathcal{A}_m$ si et seulement si \Big($m=0\,(\mathrm{mod}\, 4)$ ou $m=2(2k+1)^2-1$ pour un entier $k\geq0$\Big). Le groupe de Galois de $Q(m,m+1,x)$ est $\mathcal{A}_{m+1}$ si et seulement si $m=(2k+1)^2-1$ pour un entier $k\geq0$.
\item[(c)] Soit $p\geq3$ un nombre premier et soit $n\geq1$ un entier.  Les polyn\^omes $P(p^n, p^n+1,x)$, $Q(p^n, p^n+1,x)$, $P(p^n, p^n-1,x)$, $Q(p^n-1, p^n,x)$ sont irr\'eductibles sur $\mathbb{Q}$.
\end{enumerate}
\end{thm}

Le dernier r\'esultat que nous mettrons \`a l'honneur est nouveau et concerne la s\'erie enti\`ere associ\'ee \`a $x\mapsto (1+4x)^{-1/2}$. Dans ce cas, et contrairement au cas de la s\'erie exponentielle, les troncatures de la s\'erie ne se comportent pas du tout (dans tous les cas test\'es) comme ses approximants de Pad\'e. Plus pr\'ecis\'ement nous montrons l'\'enonc\'e suivant. Notons
\[
    \frac{P_n(x)}{Q_n(x)} = \frac{1}{\sqrt{1+4x}} + O(x^n), \deg(Q_n)\leq \frac n2,
\]
\noindent son approximant de Pad\'e d'ordre $n$.

\begin{thm}\label{thm:real_cyclo}
Les polyn\^omes $P_n$ v\'erifient les propri\'et\'es suivantes, pour $n\geq1$.
  \begin{enumerate}
      \item[(a)] $P_n(x)$ a pour racines les $\floor{\tfrac{n-1}2}$ valeurs de $x$ telles que
    \[4x+1=\left(\frac{\zeta-1}{\zeta+1}\right)^2,\text{ pour } \zeta^n=1.\]
  \item[(b)] $P_n(x)$ d\'efinit l'extension cyclotomique r\'eelle
     $\Q(\cos(\frac{2\pi}n))$.
  \item[(c)] Le groupe de Galois de $P_n$ est $(\Z/n\Z)^\times/\{\pm1\}$.
  \end{enumerate}
\end{thm}

Il est int\'eressant d'ajouter que les premi\`eres troncatures de la s\'erie enti\`ere associ\'ee \`a la fonction $x\mapsto(1+4x)^{-1/2}$ ont toutes des groupes de Galois non-ab\'eliens (voir section \ref{sec:real_cyclo}). Le texte se termine avec la section \ref{sec:real_cyclo}, qui d\'etaille la preuve du th\'eor\`eme \ref{thm:real_cyclo}. On verra que cette preuve est essentiellement bas\'ee sur une relation de r\'ecurrence fonctionnelle satisfaite par les approximants de Pad\'e consid\'er\'es.

\section{Tour d'horizon de r\'esultats connexes}\label{horizon}

Les trois paragraphes suivants pr\'esentent trois directions de recherche. La premi\`ere direction est la m\'ethode des polygones de Newton, de loin la m\'ethode la plus utilis\'ee pour obtenir des r\'esultats. La seconde direction est illustr\'ee par un th\'eor\`eme de Chambert-Loir de type Jentzsch-Szeg\"o. La troisi\`eme direction est une cons\'equence de la finitude des points rationnels sur les courbes de genre $g\geq 2$ (conjecture de Mordell, th\'eor\`eme de Faltings). 

\subsection{La m\'ethode des polygones de Newton}

Coleman (1987, \cite{Col}) donne une nouvelle preuve du th\'eor\`eme de Schur sur les troncatures $T_n$ de la fonction exponentielle (la formule explicite est rappel\'ee en (\ref{expoN})) en utilisant les polygones de Newton. Il donne en exercice page 188 les cas des polyn\^omes de Hermite et Laguerre. (On pourra aussi consulter \cite{Con}.) Rappelons ce qu'est le polygone de Newton d'un polyn\^ome $P\in{\mathbb{Q}_p(X)}$ de degr\'e $n$, o\`u $p$ est un nombre premier et $v_p$ la valuation $p$-adique. Quitte \`a diviser $P$ par une puissance de $X$, puis par $P(0)$, on peut supposer que $P(0) = 1$, de sorte que $P(X)$ s'écrive $$P(X)=1+a_1X+\cdots+a_nX^n,$$ o\`u les coefficients $a_i$ sont dans $\mathbb{Q}_p$ pour tout $i\in{1,\cdots, n}$ et $a_n\neq0$. Consid\'erons l'ensemble $S$ des points du plan d\'efinis par $S=\{  A_0=(0,0)\}\cup\{A_i=(i, v_p(a_i)), \mathrm{tels}\;\mathrm{que}\; i\in\{1,\cdots,n\}\;\mathrm{et}\; a_i\neq0 \}$. Le polygone de Newton de $P$ est alors la fronti\`ere inf\'erieure de l'enveloppe convexe de cet ensemble $S$. Il s'agit donc d'une ligne bris\'ee, r\'eunion de segments dont les extr\'emit\'es sont dans $S$.

On peut alors retrouver des informations sur la factorisation d'un polyn\^ome dans $\mathbb{Q}_p[X]$ par le calcul de son polygone de Newton. Le polygone de Newton du polyn\^ome $T_n$ (dont la formule explicite est rappel\'ee en (\ref{expoN})) pour un choix de nombre premier $p$ satisfaisant $v_p(n!)=1$ est de la forme suivante, si on note ses coefficients $(a_i)_{0\leq i\leq n}$.

\begin{pspicture}(-4,-4)(5,4)

\rput(-3.7,2.5){$v_p(a_i)$}
\rput(7,-0.4){$i$}

\psline{->}(-4, 0)(7,0)
\psline{->}(-3,-2)(-3,3)

\psdot(-3,0)
\rput(-3.4,-0.3){$(0,0)$}

\psdot(-2,0)
\rput(-2.4,+0.3){$(1,0)$}

\psdot(-1,0)
\rput(-1.4,+0.3){$(2,0)$}

\psdot(1,0)
\rput(0.8,-0.3){$(p-1,0)$}

\psdot(2,-1)
\rput(1.6,-1.3){$(p,-1)$}

\psdot(3,-1)
\rput(3,-0.7){$(p+1,-1)$}

\psdot(6,-1)
\rput(5.9,-1.3){$(n,-1)$}

\psline{-}(-3,0)(2,-1)
\psline{-}(2,-1)(6,-1)

\rput(8,1.6){Polygone de $T_n$}

\end{pspicture}

Par le th\'eor\`eme 3.1 page 100 de \cite{Cass}, on d\'eduit que $T_n=R_1 R_2$ dans $\mathbb{Q}_p[x]$, avec $\deg R_1=p$ et $\deg R_2 =n-p$. On dit que le polyn\^ome $R_1$ est pur, de pente $-1/p$. Les travaux mentionn\'es \`a pr\'esent utilisent cette m\'ethode de mani\`ere centrale. 
\\

Filaseta et Trifonov (2002, \cite{FiTr}) d\'emontrent l'irr\'eductibilit\'e des polyn\^omes de Bessel sur $\mathbb{Q}$, d\'efinis pour tout entier $n\geq0$ par $$y_n(x)=\sum_{k=0}^{n}\frac{(n+k)!}{2^j(n-k)! k!}x^k,$$ achevant ainsi la d\'emonstration d'une conjecture de Grosswald pr\'edisant cette propri\'et\'e. La m\'ethode employ\'ee est bas\'ee sur les polygones de Newton, et le c\oe ur de la preuve est une qu\^ete de nombres premiers v\'erifiant des conditions particuli\`eres.
\\

Filaseta et Lam (2002, \cite{FiLa}) \'etudient l'irr\'eductibilit\'e de polyn\^omes de Laguerre g\'en\'eralis\'es d\'efinis pour tout entier $n\geq0$ et tout nombre rationnel $\alpha$ par $$L_n^{(\alpha)}(x)=\sum_{k=0}^{n}\binom{n+\alpha}{n-k}\frac{(-x)^k}{k!},$$ o\`u on note $\binom{t}{k}=t(t-1)\cdots(t-k+1)/k!$ pour tout $t$ rationnel et tout $k$ entier positif. Lorsque le param\`etre $\alpha\in\mathbb{Q}$ est fix\'e et n'est pas un entier n\'egatif, ils montrent l'irr\'eductibilit\'e de ces polyn\^omes sur $\mathbb{Q}$, sauf pour un nombre fini de valeurs de $n$ (d\'ependant de $\alpha$). Ils utilisent une m\'ethode proche de celle de Schur, avec en plus un argument bas\'e sur une \'equation de Thue et un argument bas\'e sur des progressions arithm\'etiques de nombres premiers.
\\

Hajir (2009, \cite{Haj}) calcule le groupe de Galois de certains polyn\^omes de Laguerre g\'en\'eralis\'es, notamment quand le param\`etre $\alpha$ est un entier n\'egatif (le cas laiss\'e de c\^ot\'e par Filaseta et Lam). Il utilise aussi des polygones de Newton, ainsi que des crit\`eres d'irr\'eductibilit\'e de Coleman et de Filaseta.
\\

Akhtari et Saradha (2011, \cite{AkSa}) donnent une borne explicite $m_0$ \`a partir de laquelle les polyn\^omes de Hermite et les polyn\^omes de Laguerre (ainsi que certaines g\'en\'eralisations) de degr\'e $m\geq m_0$ sont irr\'eductibles ou presque irr\'eductibles (un polyn\^ome presque irr\'eductible \'etant simplement un polyn\^ome de degr\'e $m$ produit d'un facteur lin\'eaire et d'un polyn\^ome de degr\'e $m-1$). La m\'ethode employ\'ee est naturellement bas\'ee sur les polygones de Newton, le th\'eor\`eme des progressions arithm\'etiques de Dirichlet, et la finitude du nombre de solutions enti\`eres des \'equations de Thue.
\\

Cullinan et Hajir (2014, \cite{CuHa}) \'etudient les polyn\^omes de Legendre, d\'efinis pour tout entier $n\geq0$ par $$\mathrm{Leg}_n(x)=\sum_{k=0}^{n}\binom{n}{n-k}\binom{n}{k}\left(\frac{x-1}{2}\right)^k\left(\frac{x+1}{2}\right)^{n-k}.$$ Ils conjecturent notamment que le groupe de Galois de $\mathrm{Leg}_{2n}$ est isomorphe au produit en couronne $\mathcal{S}_2 \wr \mathcal{S}_n$ et obtiennent des r\'esultats partiels dans cette direction. La m\'ethode employ\'ee repose sur le crit\`ere de Jordan : des informations sur la taille du groupe de Galois d'un polyn\^ome $P$ peuvent \^etre obtenues en observant leur polygone de Newton associ\'e \`a un nombre premier peu ramifi\'e dans le corps de d\'ecomposition de $P$. En utilisant des congruences dites de Holt-Schur (voir \cite{CuHa} pour plus de d\'etails), ils obtiennent aussi des r\'esultats dans le cas de ramification sauvage.
\\

Shokri, Shaffaf et Taleb (2019, \cite{SST}) \'etudient les troncatures \`a l'ordre $n$ des s\'eries enti\`eres $1+\log(1-x)$, $1+\sin(x)$ et $\cos(x)$, par des m\'ethodes proches de celles de Coleman \cite{Col}, et obtiennent des conditions suffisantes sur $n$ pour d\'emontrer que le groupe de Galois de ces troncatures est aussi gros que possible.
\\

\subsection{Un th\'eor\`eme de Chambert-Loir \`a la Jentzsch-Szeg\"o}

Chambert-Loir (2011, \cite{Cha}) d\'emontre un th\'eor\`eme de type Jentzsch-Szeg\"o pour les s\'eries enti\`eres \`a coefficients dans une extension finite de $\mathbb{Q}_p$ : le degr\'e du facteur irr\'eductible unitaire de plus grand degr\'e pour une troncature de s\'erie enti\`ere (dont les coefficients satisfont une condition naturelle tr\`es g\'en\'erale) tend vers l'infini. Plus pr\'ecis\'ement, soit $p$ un nombre premier, soit $K$ une extension finie de $\mathbb{Q}_p$, notons $K[[X]]$ l'anneau des s\'eries formelles en $X$ \`a coefficients dans $K$, et pour tout r\'eel $R>0$, notons $K\{R^{-1}X\}$ l'ensemble des s\'eries $\sum_{j\geq0} a_j X^j$ de $K[[X]]$ telles que $a_j R^{j}\to 0$ lorsque $j\to +\infty$. Le th\'eor\`eme de Chambert-Loir est le suivant.

\begin{thm}
Soit $f=\sum_{j\geq0}a_j X^j\in{K\{R^{-1} X\}}$ et pour tout entier naturel $n$, notons $f_n(X)=\sum_{j=0}^{n}a_j X^j$ sa troncature \`a l'ordre $n$. Pour tout entier $d>0$, pour toute sous-suite $(n_k)_{k\geq0}$ telle que $a_{n_k}^{1/{n_k}}\to 1/R$ lorsque $k\to+\infty$, le nombre de facteurs irr\'eductibles unitaires de $f_{n_k}$ de degr\'e inf\'erieur ou \'egal \`a $d$ est $o(n_k)$. En particulier le degr\'e du plus grand facteur irr\'eductible unitaire de $f_{n_k}$ tend vers l'infini lorsque $k$ tend vers l'infini. 
\end{thm}

Ce th\'eor\`eme indique donc que l'existence d'autres exemples de s\'eries enti\`eres dont les troncatures sont des polyn\^omes irr\'eductibles est probable.

\subsection{Une cons\'equence de Mordell-Faltings}

Cullinan, Hajir et Sell (2009, \cite{CHS}) obtiennent un r\'esultat sur une sous-famille de polyn\^omes de Jacobi. Les polyn\^omes de Jacobi sont d\'efinis pour tout $n\geq0$ et tout $(\alpha,\beta)\in{\mathbb{C}^2}$ par $$P_n^{(\alpha,\beta)}(x)=\sum_{k=0}^{n}\binom{n+\alpha}{n-k}\binom{n+\beta}{k}\left(\frac{x-1}{2}\right)^k\left(\frac{x+1}{2}\right)^{n-k},$$ o\`u on note $\binom{t}{k}=t(t-1)\cdots(t-k+1)/k!$ pour tout $t\in{\mathbb{C}}$ et tout $k$ entier positif. Ce sont des polyn\^omes orthogonaux obtenus \`a partir de la s\'erie hyperg\'eom\'etrique $_2F_1$. On remarque que les polyn\^omes de Legendre sont un cas particulier des polyn\^omes de Jacobi: $\mathrm{Leg}_n(x)=P_n^{(0,0)}(x)$. Cullinan, Hajir et Sell s'int\'eressent plus pr\'ecis\'ement \`a la famille de polyn\^omes $$J_n(x,y)=(-1)^n P_n^{(-1-n, y+1)}(1-2x)=\sum_{j=0}^n\binom{y+j}{j}x^j,$$ (voir \cite{CHS} page 97 et \cite{CuHa} page 536 pour les calculs formels sur les expressions de ces polyn\^omes) en utilisant des propri\'et\'es de la courbe plane d\'efinie par $J_n(x,y)=0$. Ils montrent le r\'esultat suivant.

\begin{thm}\label{Mordell_appli}
Soit $n\geq6$ un entier naturel. Le polyn\^ome $J_n(x,y_0)$ est irr\'eductible sur $\mathbb{Q}$ pour tout $y_0\in{\mathbb{Q}}$, sauf \'eventuellement pour un nombre fini d'exceptions. De plus, si $n$ est impair, le groupe de Galois de $J_n(x,y_0)$ est $\mathcal{S}_n$ pour tout $y_0\in{\mathbb{Q}}$, sauf \'eventuellement pour un nombre fini d'exceptions. Si $n$ est pair, il existe un ensemble mince de $y_0\in{\mathbb{Q}}$ pour lesquels le groupe de Galois de $J_n(x,y_0)$ est $\mathcal{A}_n$.
\end{thm}

La m\'ethode employ\'ee suit celle de Hajir et Wong \cite{HaWo06} et repose sur la Proposition 5.17 de \cite{Mul}, que nous rappelons ici.

\begin{prop}\label{mumu}
Soit $k$ une extension finie de $\mathbb{Q}$. Soit $f(x,y)\in{k(y)[x]}$ un polyn\^ome irr\'eductible. Supposons que $f(x,y_0)$ n'est pas irr\'eductible pour une infinit\'e de valeurs $y_0\in{k}$. Alors le corps de d\'ecomposition $L$ de $f(x,y)$ sur $k(y)$ contient un corps $E$ contenant $k(y)$ tel que $f(x,y)$ n'est pas irr\'eductible sur $E$, de plus le corps $E$ est ou bien rationnel, ou bien le corps de fonctions d'une courbe elliptique avec rang de Mordell-Weil non nul.
\end{prop}

La strat\'egie de preuve du th\'eor\`eme \ref{Mordell_appli} est donc en fait bas\'ee sur le th\'eor\`eme de Faltings (conjecture de Mordell) : une courbe d\'efinie sur $\mathbb{Q}$ de genre $g\geq 2$ n'a qu'un nombre fini de points rationnels. On montre que le genre de la courbe (d\'esingularis\'ee) d\'efinie par l'\'equation $J_n(x,y)=0$ est sup\'erieur ou \'egal \`a $2$ d\`es que $n\geq6$, donc cette courbe n'a qu'un nombre fini de points rationnels, on applique alors la contrapos\'ee de la Proposition \ref{mumu}.
\\

Cullinan (2019, \cite{Cu}) \'etudie les polyn\^omes de Laguerre g\'en\'eralis\'es $L_n^{(\alpha)}(x)$, o\`u $n\geq4$ est un entier naturel et $\alpha$ est un nombre rationnel, en regardant l\`a aussi les courbes alg\'ebriques qu'ils d\'efinissent sur $\mathbb{Q}$. Il conjecture que la jacobienne d'une telle courbe n'a que tr\`es peu de points de torsion, est de rang strictement positif sur $\mathbb{Q}$, n'a pas de multiplication complexe et que ses repr\'esentations galoisiennes $\rho_\ell$ sont surjectives pour tout nombre premier $\ell\geq3$.
\\

\section{Groupes de Galois d'approximants de Pad\'e}\label{Pade}

En plus d'examiner des troncatures de certains polyn\^omes orthogonaux, nous pr\'esentons \`a pr\'esent les premiers r\'esultats, exp\'erimentaux et th\'eoriques, concernant l'arithm\'etique des approximants de Pad\'e. 

\subsection{D\'efinition}

Soient $m\geq 0 $ et $k\geq 1$ deux entiers. Soit $f:\mathbb{R}\to\mathbb{R}$ une fonction admettant un d\'eveloppement limit\'e en $0$ \`a l'ordre $m+k+1$ \`a coefficients rationnels et telle que $f(0)\neq0$. L'approximant de Pad\'e (\cite{pade}) d'ordre $(m,k)$ est la fraction rationnelle:
 \begin{equation}
 R(x)=\frac {P(x)}{Q(x)},
 \end{equation}
v\'erifiant $\deg P\leq m$, $\deg Q\leq k$ et
\begin{equation}
 f(x)=\frac{P(x)}{Q(x)}+ O(x^{m+k+1}),
\end{equation}
o\`u la notation $O(.)$ est li\'ee \`a l'approximation au voisinage de $0$. Le quotient $R$ est unique, $P\in{\mathbb{Z}[x]}$ et $Q\in{\mathbb{Z}[x]}$ le sont si on impose \`a la fraction d'\^etre r\'eduite. Dans ce qui suit, les approximations de Pad\'e que nous consid\`ererons seront dites d'ordre $n$ : l'approximation de Pad\'e d'une fonction $f$ sera le couple de polyn\^omes $(P_n,Q_n)$ \`a coefficients entiers avec $\deg Q_n \leq  \lfloor \frac{n}{2} \rfloor $ et $\deg P_n + \deg Q_n <n$ et tel que
$$ f(x)=\frac{P_n(x)}{Q_n(x)} + O(x^{n}).$$

\subsection{R\'esultats num\'eriques}

Les calculs sont men\'es avec la fonction {\tt bestapprPade()} impl\'ement\'ee dans  \cite{pari}. On pr\'esente ici des calculs obtenus en utilisant le script suivant : 
\\

\begin{center}
\begin{verbatim}
t(n) = bestapprPade(f(x+O(x^n)), n\2)
gn(n)=my(F=factor(numerator(t(n)))[,1]);F[#F];
[ polgalois(gn(n)) | n <- [n_1...n_2]]
\end{verbatim}
\end{center}

\vspace{0.3cm}

\subsubsection{Exponentielle} Commen\c{c}ons par consid\'erer les approximants de Pad\'e de la fonction exponentielle. Les premiers polyn\^omes obtenus sont irr\'eductibles. Par exemple, les cas $n=10$ et $n=13$ donnent
\begin{align*}
 & P_{10} =x^4 + 24x^3 + 252x^2 + 1344x + 3024, \\
 & Q_{10}  = x^5 - 25x^4 + 300x^3 - 2100x^2 + 8400x - 15120,\\
 & P_{13} =x^6 + 42x^5 + 840x^4 + 10080x^3 + 75600x^2 + 332640x + 665280, \\
 & Q_{13}  = x^6 - 42x^5 + 840x^4 - 10080x^3 + 75600x^2 - 332640x + 665280.
\end{align*}
Quelques calculs donnent aussi le tableau suivant, o\`u $G(P)$ d\'esigne le groupe de Galois du polyn\^ome $P$ :

\begin{center}
\begin{tabular}{|ccccccccccc|}
    \toprule
    $n$
    & $10$ & $13$ & $17$ & $18$ &$19$ &  $26$  & $34$ & $40$ & $41$& $42$\\
    \midrule
    $G(P_n)$
    & $\mathcal{A}_{4}$ & $\mathcal{S}_{6}$ & $\mathcal{S}_{8}$ & $\mathcal{A}_{8}$ & $\mathcal{S}_{9}$ & $\mathcal{A}_{12}$& $\mathcal{A}_{16}$ & $\mathcal{S}_{19}$ & $\mathcal{S}_{20}$ & $\mathcal{A}_{20}$  \\

    $G(Q_n)$
    & $\mathcal{S}_{5}$ & $\mathcal{S}_{6}$ & $\mathcal{S}_{8}$ & $\mathcal{A}_{9}$ & $\mathcal{S}_{9}$ & $\mathcal{S}_{13}$ & $\mathcal{S}_{17}$ & $\mathcal{S}_{20}$ & $\mathcal{S}_{20}$& $\mathcal{S}_{21}$  \\
    \bottomrule
 \end{tabular}   
\end{center}

\vspace{0.3cm}

On constate donc l\`a aussi une alternance de groupes sym\'etriques et de groupes altern\'es ! Ces premi\`eres constatations, propos\'ees dans une premi\`ere version de ce texte diffus\'ee sur ArXiv, ont motiv\'e un travail r\'ecent de Cullinan et Sheel \cite{CuSh}. Ils identifient les approximants de Pad\'e d'ordre $(m,k)$ de la fonction exponentielle par les formules $$P(x)=P(m,k,x)=\sum_{j=0}^m\frac{(m+k-j)!}{k!}\binom{m}{j}x^j \quad\quad\mathrm{et}\quad\quad Q(x)=Q(m,k,x)=P(k,m,-x).$$ Ils d\'emontrent alors le th\'eor\`eme \ref{expo pad}, r\'epondant ainsi partiellement \`a la question \ref{ques pad} de l'introduction. La preuve repose en bonne partie sur une remarque de Hajir : les approximants de Pad\'e de la fonction exponentielle sont en fait des cas particuliers de polyn\^omes de Laguerre g\'en\'eralis\'es ! 

\subsubsection{S\'eries logarithmiques} 

Plus amusant encore, en consid\'erant les approximants de Pad\'e d'ordre $n\leq30$ de la s\'erie
$$\frac12\log\left(\frac{1+x}{1-x}\right) = x+\frac{x^3}3+\frac{x^5}5+\dots \; ,$$
seuls des groupes hyperoctah\'edraux (groupes de sym\'etries des hypercubes) apparaissent comme groupes de Galois :
$$ G(P_n)=B_t= C_2 \wr \mathcal{S}_t \text{~et~}  G(Q_n)= B_s=C_2 \wr \mathcal{S}_s  \; ,$$ 
pour certains entiers $s$ et $t$, o\`u on note $\wr$ le produit en couronne, et $C_2$ est le groupe cyclique \`a deux \'el\'ements. Dans la suite des approximants de Pad\'e d'ordre $n\leq30$ de la s\'erie 
$$\log(1-x) = x+\frac{x^2}2+\frac{x^3}3+\frac{x^4}4+\frac{x^5}5+\dots,$$
on remarque qu'il y a apparition en alternance des groupes de type $\mathcal{S}_t$ et de type $B_t$.

\subsubsection{Fonction $x\mapsto\sin(x) + \sinh(x)$} Consid\'erons la fonction $f$ d\'efinie par
\begin{equation}
 f(x)= \sin(x) + \sinh(x).
\end{equation}
Nous sommes dans un cas o\`u les groupes de Galois qui apparaissent lorsque l'on consid\`ere les approximations de Pad\'e d'ordre $n\leq30$ et ceux qui apparaissent lorsque l'on fait la troncature de la s\'erie enti\`ere obtenue \`a partir de $f$ sont du m\^eme type. Ceux-ci sont :
\begin{equation}
 4T_{3}, 8T_{26}, 12T_{185}, 16T_{1758}, \ldots 
\end{equation}

\noindent o\`u pour $n,m \in \mathbb{N}$, la notation $nT_m$ est celle de Butler et McKay \cite{ButlerMcKay} et d\'esigne le m-i\`eme groupe transitif d'ordre $n$. Voir aussi \cite{lmfdb} pour de plus amples informations sur ces groupes. Ceci laisse penser (\`a ce stade, ce n'est qu'une remarque bas\'ee sur des calculs num\'eriques) que les groupes qui apparaissent sont des quotients de groupes de la forme $$C_4 \wr (\mathcal{S}_k\oplus C_2),$$ o\`u $C_4$ est le groupe cyclique d'ordre $4$. On peut donc conclure ce paragraphe avec enthousiasme : il y a de jolies propri\'et\'es \`a d\'ecouvrir sur cette voie !

\subsubsection{Fonction $x\mapsto(1+4x)^{-1/2}$}
L'exemple le plus abouti de notre \'etude concerne la fonction $x\mapsto(1+4x)^{-1/2}$, et fait l'objet de la section \ref{sec:real_cyclo}.

\section{Propri\'et\'es galoisiennes des approximations de la fonction $x\mapsto(1+4x)^{-1/2}$}\label{sec:real_cyclo}

Nous regroupons ici des r\'esultats concernant les approximations de la fonction $x\mapsto(1+4x)^{-1/2}$. Le premier paragraphe est une remarque num\'erique : les troncatures de la s\'erie enti\`ere d\'efinissent des groupes de Galois non-ab\'eliens dans tous les cas calcul\'es. Le second paragraphe pr\'esente l'\'etude des approximants de Pad\'e de cette m\^eme fonction, et nous donnons la preuve du th\'eor\`eme principal (le th\'eor\`eme \ref{thm:real_cyclo}) : les groupes de Galois obtenus sont ab\'eliens !

\subsection{Troncatures} \label{troncatures}
\'Ecrivons le d\'eveloppement limit\'e au voisinage de $0$
$$ \frac{1}{\sqrt{1+4x}}=U_n(x) + O(x^{n+1}). $$
Le groupe de Galois de quelques $U_n(x)$ est donn\'e dans le tableau suivant.
\\

\begin{center}
\begin{tabular}{|ccccccccc|}
    \toprule
    $n$
    & $3$ & $4$ & $5$ & $12$ &$16$ &  $20$  & $21$ & $24$ \\
    \midrule
    $G(U_n)$
    & $S_{3}$ & $\mathcal{A}_{4}$ & $S_{5}$ & $\mathcal{A}_{12}$ & $S_{16}$ & $S_{20}$& $S_{21}$ & $\mathcal{A}_{24}$  \\
    \bottomrule
 \end{tabular}   
\end{center}

Dans tous ces cas explicitement calcul\'es, les groupes de Galois sont toujours soit $\mathcal{S}_n$, soit $\mathcal{A}_n$ pour ces polyn\^omes issus de troncatures de la s\'erie enti\`ere associ\'ee, ils sont donc loin d'\^etre ab\'eliens. Nous n'avons pas encore de preuve de cette propri\'et\'e.

\subsection{Approximants de Pad\'e}

Pour la fonction $x\mapsto(1+4x)^{-1/2}$ au voisinage de $0$, l'approximant de Pad\'e d'ordre $n\geq1$ v\'erifie
$$
 \frac{ P_n(x) }{Q_n(x) } = \frac{1}{\sqrt{1+4x}} + O(x^{n}) .
$$
Un calcul rapide montre, pour $n\leq31$, que les polyn\^omes $P_n$ et $Q_n$ obtenus ne sont pas toujours irr\'eductibles, mais qu'ils semblent toujours d\'efinir des extensions ab\'eliennes ! On peut lister quelques r\'esultats dans le tableau ci-dessous, o\`u $G(P)$ d\'esigne le groupe de Galois du polyn\^ome $P$, et $C_n$ le groupe cyclique d'ordre $n$ :
\\

\begin{center}
\begin{tabular}{|cccccccc|}
    \toprule
    $n$
    & $11$ & $13$ & $17$ & $19$ & $23$ & $29$ & $31$ \\
  \midrule
    $G(P_n)$
    & $C_{5}$ & $C_{6}$ & $C_{8}$ & $C_{9}$ & $C_{11}$& $C_{14}$ & $C_{15}$  \\
    $G(Q_n)$
    & $C_{5}$ & $C_{6}$ & $C_{8}$ & $C_{9}$ & $C_{11}$& $C_{14}$ & $C_{15}$    \\
 \bottomrule
\end{tabular}
\end{center}

\vspace{0.3cm}
Outre le fait que l'on obtient des groupes de Galois ab\'eliens, on remarque que dans tous les cas explicitement calcul\'es,
les $P_n$ et $Q_n$ jouissent de propri\'et\'es de divisibilit\'e semblables
\`a celles des polyn\^omes cyclotomiques :
\begin{equation}\label{eq:divpnqn}
\text{si~} n|m \text{~alors~} P_n | P_m \text{~et si de plus
    $m/n$ est impair alors~} Q_n | Q_m .
\end{equation}

Ce n'est pas un hasard : on va \`a pr\'esent d\'emontrer
ces propri\'et\'es, ainsi que le caract\`ere cyclotomique des polyn\^omes $P_n$, comme annonc\'e dans le th\'eor\`eme \ref{thm:real_cyclo}.

Posons $y=\sqrt{1+4x}$, et consid\'erons les approximants de Pad\'e
\[
    \frac{P_n(x)}{Q_n(x)} = \frac{1}{y} + O(x^n), \deg(Q_n)\leq \lfloor\frac n2\rfloor.
\]
On d\'emontre ici que pour tout $n\geq1$, les num\'erateurs $P_n$ d\'efinissent les corps ab\'eliens r\'eels
$\Q(\cos(\tfrac{2\pi}n))$. Commen\c{c}ons par donner une expression explicite des polyn\^omes $P_n$ et $Q_n$. On fixe $P_0=0$ et $Q_0=2$.
 
\begin{lem}
Les num\'erateurs $P_n(x)$ v\'erifient la relation de
r\'ecurrence
\begin{equation}\label{recu}
  P_{n+1}(x) = P_n(x) + x P_{n-1}(x)
\end{equation}
pour tout entier $n\geq 1$ et sont donn\'es par l'expression suivante :
 \begin{equation}\label{eq:pn-expl}
  P_n(x) = \frac{(\tfrac{1+y}2)^n-(\tfrac{1-y}2)^n}y.
 \end{equation}
 
 Les d\'enominateurs $Q_n(x)$ v\'erifient la relation de
r\'ecurrence
\[
  Q_{n+1}(x) = Q_n(x) + x Q_{n-1}(x)
\]
pour tout entier $n\geq 1$ et sont donn\'es par l'expression suivante :
\begin{equation}\label{eq:qn-expl}
 Q_n(x) = \left(\frac{1+y}2\right)^n+\left(\frac{1-y}2\right)^n.
\end{equation}
 
\end{lem}

\begin{proof}
Listons les premiers termes $P_n$, pour $n\geq0$
\[
0, 1, 1, x + 1, 2x + 1, x^2 + 3x + 1, 3x^2 + 4x + 1, x^3 + 6x^2 + 5x + 1,\dots
\]
Cela laisse entrevoir la r\'ecurrence (\ref{recu}). On tire de (\ref{recu}) une expression explicite de $P_n$ en fonction du discriminant, qui est \'egal \`a $1+4x=y^2$,
et on obtient la forme annonc\'ee en \'ecrivant $P_1 = 1$ et $P_2 = 1$ d'une part, et $Q_1 = 1$ et $Q_2 = 1+2x$ d'autre part. Les expressions pour ces premi\`eres valeurs sont justifi\'ees par $(1+4x)^{-1/2}=1+O(x)=P_1/Q_1$, et de m\^eme
$(1+4x)^{-1/2}=1-2x+O(x^2)$ et $\left(\frac{1+y}2\right)^2+\left(\frac{1-y}2\right)^2=\frac{1}{4}[2+2y^2]=1+2x$, et on a bien $P_2/Q_2=(1+2x)^{-1}=1-2x+O(x^2)$. Ce calcul est de plus compatible avec notre convention pour $P_0$ et $Q_0$.

Il nous suffit donc de v\'erifier que les polyn\^omes ainsi d\'efinis
correspondent bien aux approximants
de Pad\'e : de fait, par une induction imm\'ediate, $P_n(x)\in\Z[x]$
est un polyn\^ome de degr\'e $\leq \floor{\tfrac{n-1}2}$, et
$Q_n(x)\in{\mathbb{Z}[x]}$ est de degr\'e $\leq \floor{\tfrac n2}$.

D'autre part, $1-y=O(x)$, de sorte que
\[
  Q_n(x) = \left(\frac{1+y}2\right)^n+O(x^n) = yP_n(x)+O(x^n).
\]
Ainsi $P_n/Q_n$ est bien l'approximant de Pad\'e d'ordre $n$ de $1/y$.
\end{proof}

Nous allons maintenant voir que la forme explicite donn\'ee en (\ref{eq:pn-expl}) permet de retrouver
les propri\'et\'es \'enonc\'ees dans le th\'eor\`eme \ref{thm:real_cyclo}.

\begin{proof}(du th\'eor\`eme \ref{thm:real_cyclo})

\begin{itemize}
    \item    
        La forme \eqref{eq:pn-expl} (resp. (\ref{eq:qn-expl})) donne les
        divisibilit\'es annonc\'ees en \eqref{eq:divpnqn} car si $n\mid m$, on a $a^n - b^n \mid a^m-b^m$ dans $\mathbb{Q}[a,b]$
        (respectivement si $n\mid m$ et $m/n$ est impair, alors $a^n+b^n\mid a^m+b^m$).
\item
On d\'eduit \'egalement de l'expression (\ref{eq:pn-expl}) les racines de $P_n(x)$
  en fonction de $y=\sqrt{1+4x}$
  \[
    P_n(x) = 0
    \Leftrightarrow \left(\frac{1+y}{1-y}\right)^n = 1,
  \]
 ce qui red\'emontre la relation de divisibilit\'e (\ref{eq:divpnqn}).
\item
  Par ailleurs, soit $\zeta$ une racine $n$-i\`eme de l'unit\'e,
  alors
  \[
    \frac{1+y}{1-y}=\zeta \Leftrightarrow y = \frac{\zeta-1}{\zeta+1}
  \]
  d'o\`u l'expression de $x$.
\item
$x \in\Q(\zeta)\cap\R = \Q(\zeta+\zeta^{-1})$ car
l'expression de $x$ est invariante par $\zeta\mapsto \zeta^{-1}$. Ceci implique $\mathbb{Q}[x]\subset\Q(\zeta+\zeta^{-1})=\Q(\cos(\frac{2\pi}n))$.

On a l'inclusion r\'eciproque :
si $y=\frac{\zeta-1}{\zeta+1}$ avec $\zeta=e^{i\theta}$ pour un $\theta\in{\mathbb{R}}$, alors
$y^2 = -\tan^2(\frac\theta2)$, de sorte que
\[ \zeta+\zeta^{-1} = 2\cos(\theta) = 2\frac{1+y^2}{1-y^2} \in \Q(x). \]
\end{itemize}

Ainsi
le facteur (normalis\'e en fixant le terme constant \'egal \`a 1) de plus haut degr\'e de $P_n(x)$ est
\[
  \Psi_n(x) = \prod_{d\mid n} P_d(x)^{\mu(\frac nd)},
\]
qui est un polyn\^ome irr\'eductible de degr\'e $\frac{\varphi(n)}2$
qui d\'efinit $\Q(\cos(\frac{2\pi}n))$ et
a pour groupe de Galois
\[ G(\Psi_n) = G(P_n) = (\Z/n\Z)^\times/\{\pm1\}. \] C'est la conclusion recherch\'ee.

\end{proof}

\end{document}